\documentclass[11pt,english]{article}
\usepackage[T1]{fontenc}
\usepackage[latin9]{inputenc}
\usepackage{amsmath}
\usepackage{amsthm}
\usepackage{amssymb}

\makeatletter
\theoremstyle{plain}
\newtheorem{thm}{\protect\theoremname}
\theoremstyle{plain}
\newtheorem{prop}[thm]{\protect\propositionname}
\theoremstyle{plain}
\newtheorem*{thm*}{\protect\theoremname}

\usepackage{mathrsfs}\usepackage{amsthm}\usepackage{amscd}

\usepackage{graphicx}
\usepackage{a4wide}
\@ifundefined{definecolor}
 {\usepackage{color}}{}
 
\usepackage{multirow}

\newcommand\automorphism{\mathrm{Aut}}
\usepackage{calc}
\usepackage{color, colortbl}

\newcounter{EQNR}
\renewcommand{\contentsname}{Contents\\
{\footnotesize\normalfont(A table of contents should normally not
be included)}}

\makeatother

\usepackage{babel}
\providecommand{\propositionname}{Proposition}
\providecommand{\theoremname}{Theorem}

\begin{document}

\title{A natural graph of finite fields distinguishing between models}

\author{\date{September 25, 2020}Anders Karlsson and Gaëtan Kuhn \thanks{Supported in part by the Swiss NSF.}}
\maketitle
\begin{abstract}
We define a graph structure associated in a natural way to finite
fields that nevertheless distinguishes between different models of
isomorphic fields. Certain basic notions in finite field theory have interpretations in terms of standard graph properties. 
We show that the graphs are connected and provide an estimate of their diameter. An accidental graph isomorphism
is uncovered and proved. The smallest non-trivial Laplace eigenvalue is given some attention, in particular for a specific family of 
$8$-regular graphs showing that it is not an expander. We introduce a regular covering graph and show that it is connected if and only if the root is
primitive.
\end{abstract}

\section{Introduction}

Up to isomorphism there is exactly one field of cardinality $q$ which
must be of the form $p^{k}$ for a prime $p$ and integer $k>0$,
but as is well-known the isomorphisms are not canonical. There is
therefore an issue of which concrete model of a specific finite field
to take. This is a matter of considerable practical concern because
of the many applications of finite fields where the speed of computation
is of paramount importance \cite[Ch. 11]{H13}.

Chung \cite{Ch89} and Katz \cite{Ka90} defined a family of graphs
from models of finite field mostly with the idea of producing interesting
graphs or proving interesting properties of the graphs, for example
estimating the diameter, using deep results in number theory notably
on character sums or the Lang-Weil theorem. In this direction we refer
to a more recent paper \cite{LWWZ14} for a generalization of their
construction. Prior to the well-cited papers of Chung and Katz, there
are other types of graphs associated with fields, the oldest being
the Paley graphs where one starts from the additive group of the field
and take as generating set for the corresponding Cayley graphs the
elements which are not squares, thus making a connection to the topic 
of quadratic reciprocity. Very recent general constructions in yet different
directions are the functional and equational graphs in \cite{K16, MSSS20}.

The present paper instead mainly aims at studying the fields themselves from associated
graph structures. Our starting point is to take a field $K$ of characteristic
$p$ and cardinality $q=p^{k}$ given as $F_{p}[x]/(f(x))$ where
$f$ is an irreducible polynomial of degree $k$. We will refer to
this as a model $K_{f}$ for the finite field of $q$ elements. We
asked ourselves the following question: \emph{Can one define a graph
associated to $K_{f}$ in a canonical way, so that field automorphisms
are graph automorphisms but which nevertheless can distinguish the
different models of two isomorphic fields?} In other words, field
automorphisms should be graph automorphisms but field isomorphisms
should not always be graph isomorphisms. The answer turns out to
be \emph{yes} (see Proposition \ref{prop:auto} and the table below),
although the theoretical aspects of this phenomenon are only partially
clear. The relations between the properties of our graphs and questions
of efficient computing are also left for future investigation. For
example, Conway polynomials have been used to define finite fields
in some computer algebra systems, are their properties in any way
reflected by the corresponding graphs? 

One relevant property could be the size of the graph automorphism
group. Most of our graphs have as automorphism group just the field
automorphisms with one extra involution. For example, out of the 150
irreducible monic polynomials of degree 4 in $F_{5}[x]$, only eight
of them have a larger automorphism group than 8 elements. The exceptional
orders are $32$, $32768$ and $\approx3\cdot10^{45}.$ This latter
really large automorphism group appears for the graphs associated
to the fields
\[
F_{5}[x]/(x^{4}+2)\textrm{ and \ensuremath{F_{5}[x]/(x^{4}+3)} }
\]
which moreover are isomorphic as graphs. See the figures below, where
the vertices are placed on a circle, and where 
the yellow parts correspond to the edges coming from addition and the blue from multiplication.

\begin{figure}[h]
\begin{minipage}[b]{0.50\textwidth} \label{fig1}
  		\begin{center}
			\includegraphics[scale=0.5]{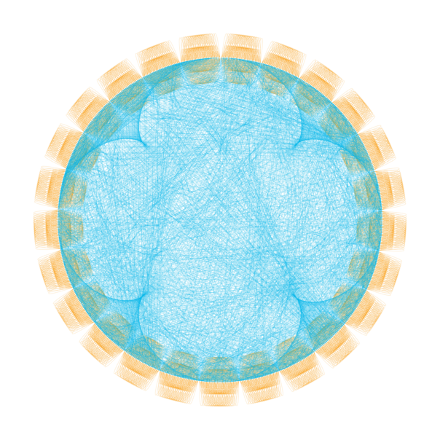}  \caption{$F_{5}[x]/(x^{4}+2)$}
		\end{center}
	\end{minipage}
	\begin{minipage}[b]{0.50\textwidth}
		\begin{center}
			\includegraphics[scale=0.5]{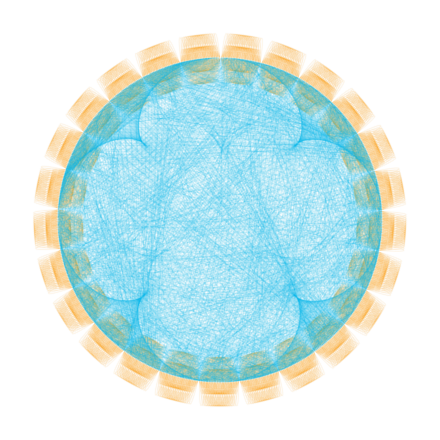} \caption{$F_{5}[x]/(x^{4}+3)$}
		\end{center} 
	\end{minipage} 
\end{figure}

Our graphs seem most of the time to distinguish between distinct models,
see the table at the end of the paper. 


As illustrated above there are exceptional isomorphisms, see also 
Proposition \ref{prop:isom}  where we give one proof of this phenomenon.
In fact, all the cases we know of can be explained with an isomorphism
coming from a reciprocal polynomial, mapping $x\mapsto x^{-1}$. This
is consistent with that the graph isomorphism classes seem to contain
at most two elements. In any case, the following general question
remains unanswered: \emph{Are all isomorphic graphs isomorphic via
a field isomorphism?} A positive answer would be remarkable.

The following result established below summarizes some basic properties:

\begin{thm*}
The graph associated to a field $K_f$ of $p^k$ elements with $k\geq 2$ is
connected, Eulerian, has girth $2$ and diameter at most $2p(2k+1)-2k-4$.
\end{thm*}

In contrast to Cayley graphs of abelian groups, the spectrum of our graphs 
associated to finite fields is non-trivial to determine due to that both addition and 
multiplication are taken into account. Nevertheless, for $F_p[x]/(x^2+1)$ with prime $p\equiv 3 \mod 4$
we 
could find an explicit part of the Laplace spectrum sufficient to conclude:

\begin{thm*}
The graphs of $F_p[x]/(x^2+1)$ with prime $p\equiv 3 \mod 4$ as $p\rightarrow \infty$ is not an expander family.
\end{thm*}
 
In the last section we inroduce a natural 
covering space of our graphs and show that it is connected if and only if $x$ is
primitve.
 

It is a pleasure to thank P{\"a}r Kurlberg for several helpful comments.

\section{A graph structure associated to finite fields}

We will not try to survey all possible constructions, beyond having
mentioned some of them briefly in the introduction, and instead we
go directly to what we suggest here. The basic data is a finite field $K_{f}$
of cardinality $q=p^{k}$ given as 
\[
K_{f}=F_{p}[x]/(f)
\]
 where $f$ is an irreducible polynomial of degree $k$ in $F_{p}[x]$.
Let $S$ be the subset 
\[
\left\{ x,x^{p},x^{p^{2}},...,x^{p^{k-1}}\right\} 
\]
of $K_{f}$, which is the conjugates of (the equivalence class of)
$x$, or in other words the orbit of $x$ under the Frobenius automorphism $y\mapsto y^{p}$.
Recall that the Frobenius map is a generator of the group of field
automorphism of $K_{f}$ which is the cyclic group of order $k$.

We now define our graph, which in a natural way is a directed graph
but we will mostly choose to forget the orientation. The vertex set
is the set $K_{f}$. The edges are of two types, corresponding to
the two field operations. For each vertices $y,z\in K_{f}$ and $s\in S$
we have a corresponding edge if
\[
z-y=s
\]
or 
\[
zy^{-1}=s.
\]
Note that for the latter type of edges neither $y$ nor $z$ can be
$0.$ This will have as a consequence that the graph is not regular
(i.e. not constant vertex degree). (An alternative definition is to instead look at the 
equation $z=sy$, which would give rise to several loops at $0$, on the other hand the graphs
would be regular.)  Each edge as above is directed
from $y$ to $z$. We denote the resulting directed graph $\overrightarrow{X_{f}}$
and undirected graph $X_{f}$.

\noindent {\bf Examples.} See the two figures, where the shades of orange correspond
to the additive edges, and the shades of blue the multiplicative ones.
(The shading distinguishes between the different elements in $S$).

\begin{figure}[h!]
\includegraphics[scale=0.4]{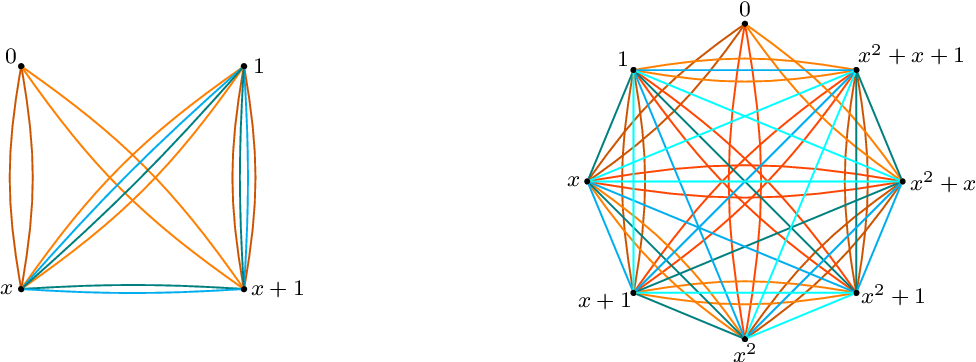} \centering \caption{$F_2[x]/(x^2+x+1)$ and $F_2[x]/(x^3+x+1)$} 
\end{figure}

As in the pictures we could consider the graph $X_{f}$ with the added
structure of coloring. Or, we could consider two subgraphs of $X_{f}$, 
namely the additive one (orange) and the multiplicative one (blue)
with $0$ removed. We call these the additive respectively the multiplicative
graph associated to $K_{f}$. While the graph $X_{f}$ is
not regular because of
the exceptional vertex $0$, the additive and multiplicative graphs
are however regular, in fact the constant vertex degree equals twice
the cardinality of $S,$ that is $2k$.

After a presentation by the second author, Pierre de la Harpe pointed out that this 
construction could be considered more
generally for rings. (One would then of course formulate 
the edge condition
for multiplicative edges without the inverse, i.e. $y$ and $sy$ are connected by a
multiplicative edge.)

\section{Automorphisms and isomorphisms}

For the purposes of this article we see the collection of all graph morphisms as being a subset of maps between vertex sets. This means for example 
that we do not distinguish the identity map from the automorphism which fixes all vertices but permutes a double edge.

The fundamental property we wanted at the outset was that our graph construction
is natural in the sense that every field automorphism is also a graph
automorphism: 
\begin{prop}
\label{prop:auto}Every field automorphism of $K_{f}$ defines also
a graph automorphism of $X_{f}$ and $\overrightarrow{X}_{f}$. 
\end{prop}

\begin{proof}
Every automorphism $\phi$ is a power of the Frobenius map, therefore
it leaves the set $S$ stable, indeed permuting it. Being a field
automorphism $\phi$ respects all the field operations, so that for
every edge defined by $y,z\in K_{f}$ and $s\in S$, it holds that
\[
\phi(z)-\phi(y)=\phi(z-y)=\phi(s)\in S
\]
and
\[
\phi(z)\phi(y)^{-1}=\phi(zy^{-1})=\phi(s)\in S.
\]
Moreover, obviously $\phi(0)=0.$ All this means precisely that $\phi$
is a graph automorphism: it permutes the vertices in such a way that
edges map to edges, and in the present case $\phi$ even respects
the orientation of the edges. 
\end{proof}
On the other hand, as will be seen, field isomorphisms between distinct models 
are not necessarily
graph isomorphisms, indeed it seems typically not to be the case.
This is possible in particular because the generating sets $S$ may
not correspond under isomorphisms between the two models of the field. 

As mentioned in the introduction, sometimes the graph automorphism group
is much larger that the field automorphism group, see also 
the table in the appendix. But most of the time, the graph automorphisms are just double
in number compared with the field automorphisms thanks the the following involution (which
however is trivial in characteristic 2):
\begin{prop}\label{propinvolution}
The map $y\mapsto -y$ is an automorphism of $X_{f}.$
\end{prop}

\begin{proof}
The map clearly is a bijection on the level of vertices. Moreover,
for every edge defined by $y,z\in K_{f}$ and $s\in S$, it holds
that
\[
-z-(-y)=(y-z)=-s
\]
and so $(-y)-(-z)=s$ so the vertices are connected by an edge (but
here the orientation of the edge is reversed, so it is not an automorphism
of the directed graph). In addition,
\[
-z(-y)^{-1}=zy^{-1}=s.
\]
This shows that the map $y\mapsto-y$ is a graph automorphism.
\end{proof}

With computer experiments using the software Sage it seems that typically
the graphs are non-isomorphic for two different irreducible polynomials
of the same degree over the same finite prime field. See the table at the end 
taken from \cite{Ku20}. As can be seen, sometimes there are
however exceptional isomorphisms. We noticed that for at least some
of these examples the graph isomorphism comes from a field isomorphism
of the following kind: every element $a(x)$ in $K_{f}$ is sent to
$a(t^{-1})$ in $F_{p}[t]/(g(t))$, and the elements of $S$ in $K_{f}$
are mapped to the corresponding set of generators in $K_{g}$ or their
inverses additively and multiplicatively. Here is a proof of the first
non-trivial graph isomorphism appearing in the tables:

\begin{figure}[h!]
\includegraphics[scale=0.4]{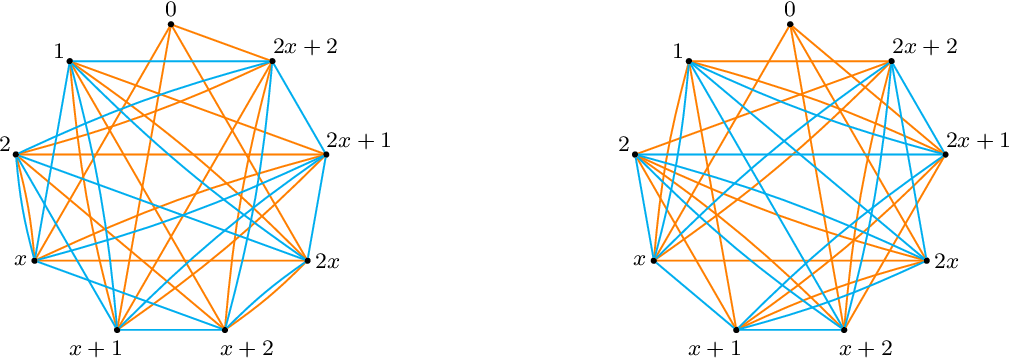} \centering \caption{$F_{3}[x]/(x^{2}+x+2)$ and $F_{3}[x]/(x^{2}+2x+2)$} 
\end{figure}

\begin{prop}
\label{prop:isom}The graphs associated to the fields $F_{3}[x]/(x^{2}+x+2)$
and $F_{3}[x]/(x^{2}+2x+2)$ are isomorphic.
\end{prop}

\begin{proof}
To avoid confusion we use the variable $t$ in the second field $F_{3}[t]/(t^{2}+2t+2)$.
First we notice that the map $\phi$ sending an element $a(x)$, a
polynomial in $F_{3}[x]$ of degree less than 2, to $a(t^{-1})$ is
indeed a field isomorphism. Here we suppress notation for equivalence
classes modulo the polynomials. The map is obviously a bijection preserving
the prime field. Moreover it is a ring isomorphisms considering that
we are merely doing a substitution $x=t^{-1}.$ For it to be well-defined,
we need one calculation. First we observe that $t^{-1}=t+2.$ Then
we calculate
\[
t^{-2}+t^{-1}+2=(t+2)^{2}+(t+2)+2=t^{2}+4t+4+t+2+2=t^{2}+2t+2
\]
which is precisely what is needed for the map to descend to an isomorphism
on the quotient fields.

Now we calculate 
\[
S=\left\{ x,x^{3}\right\} =\left\{ x,2x+2\right\} 
\]
and 
\[
\left\{ t,t^{3}\right\} =\left\{ t,2t+1\right\} .
\]
The generators $S$ of the first field is mapped to 
\[
\left\{ t^{-1},2t^{-1}+2\right\} =\left\{ t+2,-t\right\} =\left\{ -(2t+1),-t\right\} .
\]
This shows already that the additive edges are mapped to additive
edges (orientation reversed). That is, an edge $a(x)=b(x)+s$ is mapped
to $a(t^{-1})=b(t^{-1})+\phi(s)$, then this is $a(t^{-1})+(-\phi(s))=b(t^{-1})$
and $-\phi(s)\in\left\{ t,t^{3}\right\} $. 

Let us finally study the multiplicative edges which also will be orientation
reversed, meaning that an edge $y=zs$ is instead $ys^{-1}=z$. We
observe that
\[
\left\{ (t+2)^{-1},(-t)^{-1}\right\} =\left\{ t,-(t+2)\right\} =\left\{ t,2t+1\right\} =\left\{ t,t^{3}\right\} ,
\]
which is exactly what is required.
\end{proof}

It is interesting that the above graph isomorphism comes from the field structure. 
We do not yet know of a situation where this is not the case, that is, when two
graphs are isomorphic but no graph isomorphism is also at the same time a field 
isomorphism.

The proposition generalizes and a proof analysis would give a general
theorem. But since at present time we do not have a precise conjecture
for when the graphs are isomorphic or not, we leave this exercise
for now, except for drawing the attention to the following notion. 
Given a polynomial $f(x)=a_nx^n+...+a_0$. Recall that the 
\emph{monic reciprocal polynomial} is by definition 
$g(x)=a_0^{-1}x^nf(x^{-1})$.

\noindent {\bf Examples.}  The pair of polynomials in the proposition are monic reciprocal of each other. Same goes for $x^4+x^2+2$ and $x^4+2x^2+2$ in characteristic $3$ as well as $x^2+2$ and $x^2+3$ in characteristic $5$. These pairs moreover have isomorphic graphs. On the other hand, the reciprocal polynomials $x^3+2x+2$ and $x^3+x^2+2$ in characteristic $3$ do not have isomorphic graphs.
\vspace{5pt}

One can consider certain subgraphs, that could be called \emph{core graphs}, which 
are the subgraphs on all vertices but only edges defined by one fixed element of $S$, for example $x$.
Since the elements in $S$ are conjugates one can see that the core graphs of a given model are isomorphic. One can also verify that for reciprocal polynomials, their core graphs are isomorphic. But as the latter among the 
listed example above shows, this may not extend to a graph isomorphism of the full graph. Another example now in characteristic 2, the polynomials $x^5+x^4+x^2+x+1$ and $x^5+x^4+x^3+x+1$ are both primitive, normal and reciprocal to each other, still their graphs are not isomorphic. 

A further observation from the table is that so far in characteristic
2, there are no non-trivial isomorphisms among the cases listed in the table. 
The role of characteristic $2$
implying that $-1=1$ already proved special when looking at the automorphism
group in the previous section. Moreover, one only sees pairs of isomorphic
graphs in the tables, so far no three isomorphic models. Although we think it is 
too risky to conjecture
that all isomorphic graphs would be of the type explained in
the previous proposition, at least one cannot help to ponder this
possibility. 

If one prefers to instead investigate the directed graphs or the two
partial graphs (addition and multiplication) this picture roughly
remains the same: some models are distinguished, but still there are some
unexplained graph isomorphisms. 

\section{Connectivity properties}

One of the most basic property of a graph is whether it is connected
or not. The graphs here are connected (this uses that $K_f$ is a field and not merely
a ring):
\begin{thm} \label{prop:graphconnected}
The graphs $X_{f}$ and $\overrightarrow{X}_{f}$ are connected, respectively
strongly connected. They are moreover Eulerian in respective senses. The diameter of $X_{f}$
is less than $2p(2k+1)-2k-4$, while that of $\overrightarrow{X}_{f}$ is less than 
$(p-1)(k^2+4k+1)+k$.
\end{thm}

\begin{proof}
Given an arbitrary element in the field $u_{0}+u_{1}x+...+u_{k-1}x^{k-1}$,
we will connect it to $0$ with a directed path, and from $0$ to
this element. We start with the latter. First, $0$ is connected to
$x$ since $x\in S$. Then $x$ is connected to $2x$, and we continue
in this additive direction until reaching the vertex $u_{k}x.$ From
there we take a step in the multiplicative direction, from $u_{k-1}x$
to $u_{k-1}x^{2}$. Now again working additively with $x$ we connect
this to $u_{k-2}x+u_{k-1}x^{2}.$ We continue this procedure until arriving
at $u_{0}x+u_{1}x^{2}+...+u_{k-1}x^{k}$. Let $n$ be the order of
$x$ in the multiplicative group, so $x^{n}=1$. Now we take $n-1$
multiplicative steps with $s=x$ and arrive at $u_{0}x^{n}+u_{1}x^{n+1}+...+u_{k-1}x^{n-1+k}$
which finally equals the desired end vertex $u_{0}+u_{1}x+...+u_{k-1}x^{k-1}$. This is
a valid path also in the directed graph.

To prove the connectedness for the directed graph we need also to go 
from $u_{0}+u_{1}x+...+u_{k-1}x^{k-1}$ to $0$. The former vertex is connected
to $u_{0}x+u_{1}x^{2}+...+u_{k-1}x^{k}.$ Now keep adding $x$ until
we are at $u_{1}x^{2}+...+u_{k-1}x^{k}.$ Multiply by $x^{n-1}$ until
reaching $u_{1}x+...+u_{k-1}x^{k-1}$. Now repeat this procedure until
arriving at $0$. This proves the asserted connectedness properties.

The fact that they are moreover Eulerian comes from a well-known fact
we need in addition have that the vertex degrees are even which we
have, respectively that at every vertex the outgoing degree equals
the incoming degree, this we also have (notice that also $0$ satisfies
this). 

For the diameter estimates we first consider the undirected graph and the path from $0$
to $u_{0}+u_{1}x+...+u_{k-1}x^{k-1}$. The path joining $0$ to $u_{k}x$ is at most $p-1$ steps long.
Then one multiplicative step is taken and the the procedure is repeated $k$ times. This 
gives a path of length at most $k\times (p-1)+(k-1)$. Now we consider the multiplication by $x^{n-1}$. 
Here $n$ is the order of $x$ and thus divides $p^k-1$. Taking advantages of all elements in $S$ we 
expand $n$ in base $p$. The sum of digits is the length of this path. This sum is at most
$$
(\log_p(n-1)+1)(p-1)
$$
which in turn is strictly less than $(k+1)(p-1)$. All taken together the diameter must therefore be less than
$$
2((k+1)(p-1)+k(p-1)+(k-1))=2p(2k+1)-2k-4.
$$   

Finally, for the directed graph we have $(k+1)(p-1)+k(p-1)+(k-1)$ for the path going out from $0$ and then
for the second path going in to $0$ we count $1+k(p-1)+k(\log_p(n-1)+1)(p-1)$. Thus all taken together we obtain:
$$
(p-1)(k^2+4k+1)+k,
$$
which is an upper bound of the diameter in the directed case.
\end{proof}

Recall the standard notions of $x$ being \emph{primitive} if it generates
the group of units $K_{f}^{\times}$ and it is $normal$ if its conjugates
(i.e. the set $S$) form a basis for $K_{f}$. The
primitive normal basis theorem (due to Carlitz, Davenport, and Lenstra-Schoof)
asserts that there exists $f$ for which $x$ is both primitive and normal.
One interest in normal bases is that they are used in practice for
efficient numerical exponentiation in finite fields. For more about
these field theoretical aspects we refer to \cite{H13}. We connect
to our graphs:
\begin{prop}
Let $K_{f}$ be a finite field and $X_{f}$ its graph. The additive
subgraph is connected if and only if $x$ is normal. The multiplicative subgraph
is connected if and only if $x$ is primitive.
\end{prop}

\begin{proof}
This is basically clear from the definitions. The element $x$ is
primitive precisely when all elements in $K_{f}^{\times}$ is a power
of $x$ which is the same as that the multiplicative subgraph is connected.
The set $S$ has the cardinality of a basis, and if every element
can be written as a linear combination of these elements, then $x$
is normal, but this is also the same that $0$ can be joined by a
path of additive edges to every element of $K_{f}$, thus the graph
is connected precisely when $x$ is normal.
\end{proof}

\noindent {\bf Example.} In characteristic 2, the polynomials $x^3+x^2+1$ and $x^3+x+1$
form a pair of reciprocal polynomials. In the latter, the additive graph is connected,  
in other words $x$ is normal, but sketching the graph of the former one observes 
that the additive graph is not connected, thus $x$ is not normal. 
(This is in contrast with primitivity which is preserved taking the reciprocal 
polynomial.) Clearly the 
graphs are therefore not isomorphic, and incidentally it explains why the 
graph automorphisms of the first polynomial is so large: the connected component 
of the additive graph not containing $0$ is a complete graph on $4$ vertices. This 
has the symmetric group on four letters as isomorphism group, which has order $24$.
Also the multiplicative subgraph is a complete graph, which implies that these 
automorphisms can be extended to the full graph (acting trivially on the other connected 
additive component), giving $6\times 24=144$ as the order of the automorphism group.   
\vspace{5pt}

\noindent {\bf Example.} To understand the definitions one can even consider
 the trivial example $k=1$, say 
$K=F_p[x] /(x-1)$. This means that $S=\{ 1 \}$ and the additive graph is a circle, having 
a fair amount of automorphisms. On the other hand the multiplicative graph has a loop at 
each vertex except $0$ (basically a matter of convention in the definition). This means that for the total graph, rotations are not automorphisms 
since $0$ needs to be fixed. So the only remaining graph automorphism is the one given
in Proposition \ref{propinvolution}, hence the graph automorphism group is the cyclic group of order $2$ if $p>2$, while in case $p=2$ both the graph and field automorphism groups are trivial. One could instead consider $x-a$, giving other graphs 
with the additive and multiplicative subgraphs connected or not. 
\vspace{5pt}

It is natural to wonder about girth, that is the length of the shortest closed path. For example this is studied in \cite{Ka90}  for the graphs considered there. It translates into expressing every element in terms of the elements in $S$ in a minimal (non-trivial) fashion. It is obvious that in our case the girth is at most the characteristic $p$ since $a=a+px$. But could it be smaller? Yes, in fact if $k\geq 2$ then there is always a square $x,x+x^p,2x+x^p,2x,x$. And in the trivial case $k=1$ there are self-loops so the girth is $1$. So the girth is at most 4 in any case. But it can be even smaller, for example in the $p=3$ examples above it is visibly 2, using one multiplicative and one additive edge. In fact this is the general picture:

\begin{prop}
The graphs $X_f$ has girth 2 whenever $k\geq 2$.
\end{prop}
\begin{proof}
Consider the equation $a+x=a\cdot x$. If it has a solution, then this provides a closed path of length 2. Since $x\neq 1$ in view of $k>1$, we can solve for $a$, namely
$$
a=x(x-1)^{-1}.
$$ 
There are no closed paths of length $1$, that is, self-loops at a vertex (thanks to that we chose not to include the multiplicative edges from $0$ in the definition of $X_f$). To see this, additively we would have $a=a+x^n$, which cannot happen since $x\neq 0$ and also $a=ax^{p^m}$ has no solution for $a\neq 0$ since $x^{p^m}\neq 1$ when $k>1$.
\end{proof}

\noindent {\bf Example.} In $F_{3}[x]/(x^{2}+x+2)$, see Figure 5, we have $a=2x+2$ giving rise to the closed path from $2x+2$ to $2$ and back. 

\section{Spectral estimates}

The Laplacian of a finite (undirected) graph is the operator on functions $g$ on the vertices, defined by:
$$
\Delta g (y)= \sum g(y)-g(z),
$$ 
where the sum is over edges having one endpoint at $y$ and $z$ denotes the other endpoint. (Note that any loops at $y$ play no role for the definition.) do It is well-known that this operator (or matrix) is symmetric and positive semi-definite. The smallest eigenvalue is $\lambda_0 =0$ with the constant functions as corresponding eigenvectors, and the next smallest $\lambda_1$ is strictly positive if and only if the graph is connected. Indeed, this eigenvalue is an important measure of connectivity. The larger the gap to $0$ the better it is connected, called expansion property. It is highly desirable to have a sequence of $d$-regular graphs with a spectral gap that stays bounded, such sequences are called expanders. We refer to \cite{Sp19} for background and further references on these topics.

For a general graph it is typically difficult to determine the spectrum of its Laplacian. It seems to be the same for our graphs, in spite of that Cayley graphs of abelian groups have an explicit spectrum. The difficulty for fields comes from the interaction of the addition and multiplication operations. 

Kurlberg suggested the following family of graphs in this context. Consider the polynomial $f(x)=x^2+1$ and primes $p$ which are congruent to $3$ modulo $4$. One sees that $f(x)$ is then irreducible giving rise to fields $K_f$ with $p^2$ elements. The root to adjoin we denote like in complex analysis by $i$, thus we have $i^2=-1$. The generating set for the graph is easily calculated to be $\{-i,i\}$. The Laplacian of the corresponding graph $X_f$ is
$$
\Delta g(y)=8g(y)-2g(y+i)-2g(y-i) -2g(yi)-2g(-yi).
$$ 
Note that this formula is valid even at $y=0$ where there are no multiplicative edges and the vertex degree is $4$ (since when $y=0$ also $yi=0$). In order to conform with the most standard definition of expander we could add loops at $0$ to make the graphs $8$-regular (this procedure does not change the eigenvalues of the laplacian as already remarked).  Computer calculations seemed to indicate to us that this sequence is not an expander and we are in fact able to establish this with a proof:

\begin{thm}
The family of $8$-regular graphs coming from the fields $F_p[x]/(x^2+1)$ with prime $p\equiv 3 \mod 4$ as $p\rightarrow \infty$ is not an expander.
\end{thm}

\begin{proof}
We use the notation introduced above, and denote a general element $v+iw$ with $v,w\in F_p$. Note that $(v+iw)i=-w+iv$. Let $l$ be an integer between $1$ and $p-1$.  Define $e(u)=\exp(2\pi i lu/p)$ which is a well defined function for $u\in F_p$. In the hope of finding some explicit eigenfunctions of the Laplacians we let
$$
g(v+iw)=e(v)e(w)+e(-v)e(w)+e(-v)e(-w)+e(v)e(-w).
$$ 
We first calculate $g((v+iw)i)$ which gives
$$
g(-w+iv)=e(-w)e(v)+e(w)e(v)+e(w)e(-v)+e(-w)e(-v)=g(v+iw).
$$
Similarly we get that $g((v+iw)(-i))=g(v+iw)$.

Next we develop, using $e(w\pm 1)=e(\pm 1)e(w)$,
$$
g(v+i(w+1))+g(v+i(w-1))=
$$
$$
e(1)e(v)e(w)+e(1)e(-v)e(w)+e(-1)e(-v)e(-w)+e(-1)e(v)e(-w)+
$$
$$
e(-1)e(v)e(w)+e(-1)e(-v)e(w)+e(1)e(-v)e(-w)+e(1)e(v)e(-w).
$$
Notice that this equals $(e(1)+e(-1))(e(v)e(w)+e(-v)e(w)+e(-v)e(-w)+e(v)e(-w))$.

In summary we therefore get
$$
\Delta g(v+iw) = (4-2e(1)-2e(-1))g(v+iw)=8\sin^2(\pi l/p) g(v+iw)
$$
which shows as desired that $g$ is an eigenfunction. The corresponding eigenvalue is $8\sin^2(\pi l/p)$. 

For a fixed $l$, for example $l=1$, as $p$ goes to infinity, this eigenvalue is approximately equal to $8\pi^2/p^2$, which tends to $0$. This of course means that $\lambda_1 \rightarrow 0$ and so therefore this sequence of graphs is not an expander.
\end{proof}

The computer calculations alluded to above indicate that the eigenvalue here determined seems to be of the same order of magnitude as $\lambda_1$. In the general case one can obtain certain inequalities:

\begin{prop}
Given a graph $X_f$ of a finite field of $p^k$ elements. The first non-trivial eigenvalue $\lambda_1$ satisfies
$$
\lambda_1\geq \frac {1}{p^{k+1}(2k+1)}.
$$
In case $x$ is normal, 
$$
\lambda_1\geq 4\sin^2(\pi/p).
$$
\end{prop}

\begin{proof}
Let $D$ denote the diameter of the graph $X_f$, which has $p^k$ number of vertices. Inequalities between the diameter and the first non-trivial eigenvalues appear in \cite{Ch89}. For example, from Lemma 10.6.1 in \cite{Sp19}  one knows the inequality
$$
\lambda_1\geq \frac {2}{D(p^k-1)}
$$
This gives together with our estimate for the diameter in Theorem \ref{prop:graphconnected} that
$$
\lambda_1\geq \frac {1}{p(2k+1)(p^k-1)}
$$
which shows the first inequality.

For the second statement, when $x$ is normal the additive graph is a discrete torus of side lengths $p$. These graphs have well-known explicit spectrum. In particular the smallest eigenvalue is $4\sin^2(\pi/p)$. As is well-known, adding edges an only increase the eigenvalues, see for example Corollary 5.2.2 in \cite{Sp19}. Therefore the claimed assertion follows, since for the torus we have identified the smallest non-zero eigenvalue. When adding the multiplicative edges to get the full graph we can never have any eigenvalue smaller than that (also since the trivial eigenvalue stays $0$).
\end{proof}

\section{A regular covering space}

\begin{minipage}[b]{0.6\textwidth} 
It is natural to search for simple invariants, ideally complete,
that detect the isomorphisms classes of the graphs. With this motivation
in mind, let us here describe a covering graph that we find interesting
and that might moreover be useful for example in the study of spectral
properties of the graphs $X_{f}$. 

We define a natural covering space (graph) $C_{f}$ of our graph $X_{f}$.
The vertex set is the set $K_{f} \times K_{f}^\times$.
For each vertices $(y,z), (y',z') \in K_{f} \times K_{f}^\times$ 
and $s\in S$ we have corresponding edges if 
\[ 
y'-y=s, z=z'
\]  or 
\[ y'y^{-1}=z'z^{-1}=s. \]
\end{minipage}
\begin{minipage}[b]{0.4\textwidth}
	\begin{center}
		\fbox{\includegraphics[scale=0.4]{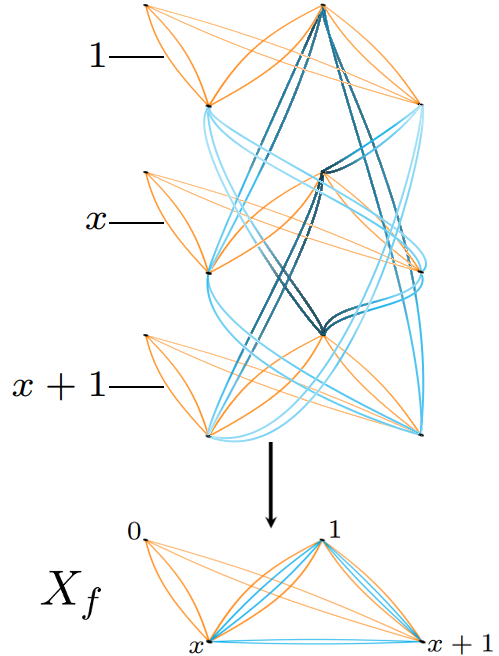} }
	\end{center}
\end{minipage}
Note that for the latter type of edge neither $y$ nor $y'$ can be $0.$ 

\begin{prop}
The graph $C_{f}$ is a regular covering space of $X_{f}$.
\end{prop}
\begin{proof}
The map $\pi : C_{f} \rightarrow X_{f}$ given by $(y,z) \mapsto y$ is clearly a surjective graph morphism. It is then also clear that it is 
a covering map. 

From above discussions it follows that field automorphisms also defines automorphisms of the graph $C_f$. There are also many covering transformations of the following kind:
 $K_{f}^\times$ acting on $C_f$. 
  Given an element $a\in K_{f}^\times$ we define $F_a :C_f \rightarrow C_f$ via
 $$
 F_a(y,z)=(y,az).
 $$
It is immediate that $\pi\circ F_a =\pi$. We need to verify that it is a graph automorphism, for this it remains to see that edges are mapped to edges. This is easily done:
there is an edge between $(y,z)$ and $(y + s,z)$ (respectively between $(y,z)$ and $(ys,zs)$) if and only if there is an edge between $(y,az)$ and $(y + s,az)$ (respectively between $(y,az)$ and $(ys,azs)$).
 
The group of these transformations clearly acts transitively on the fibers $K_{f}^\times$ of the covering. Thus our covering graph is regular as was to be shown.
\end{proof}

Note that while these graphs are \emph{regular}  in the sense of covering space theory, they are not in the sense of graph theory.  Alternative words in use for \emph{regular} in the covering space context are \emph{normal} or \emph{Galois}, but both of these terms have different meanings in the theory of fields.

\begin{prop}
The graph $C_f$ is connected if and only if $x$ is primitive.
\end{prop}

\begin{proof}
If $x$ is not primitive, it is not possible to join certain levels $(*,a)$ because the powers of $x$ are not enough. Hence $C_f$ is not connected in this case. 

Assume now that $x$ is primitive. Every non-zero element can thus be written $x^m$. 
We need to show that we can join $(0,1)$ to the vertices $(0,x^l)$ and $(x^m, x^l)$ for any $l$ and $m$.
We describe a path that corresponds to a sequence of 
addition and multiplication by $x$ in the model field $K_f$. This path can be reversed using characteristic $p$ and
the order of $x$ or just by adding $-x$ and multiplying by $x^{-1}$ if we choose to forget orientation.

It is enough to show that $(0,1)$ can be joined to $(x^n,1)$ for
any $n$ since then we can link $(0,1)$ to $(1,1)$. So we can reach any level by using 
multiplication by $x$ appropriate times. Then going from $(x^l,x^l)$ to $(0,x^l)$ is
just reversing the path between $(0,1)$ and $(x^l, 1)$.

For this, consider first the following path. From $(0,1)$ we take additive step by $x$ to $(x,1)$.
After that multiply enough times by $x$ to arrive at $(1,x^{-1})$. 
Now add $x$, and multiply by $x$ leading us to $(x+x^2,1)$ via $(1+x,x^{-1})$. 
Finally adding $x$ enough times using characteristic $p$ we arrive at $(x^2,1)$. 
This path can easily be reversed in a natural way.

The latter path is the first in the induction, assume we have $(0,1)\rightarrow (x^n,1)$. 
Then go to $(1,x^{-n})$, add $x$ and using multiplication $n$ times to arrive at $(x^n+x^{n+1},1)$. 
Finally, join $(x^n+x^{n+1},1)$ to $(x^{n+1},1)$ by following backward the path from $(0,1)\rightarrow (x^n,1)$.
We described all required paths to prove the connectedness of $C_f$.
\end{proof}

\newpage

\section{Appendix}

\vspace{-5pt}
Here is a table extracted from the unpublished memoir \cite{Ku20}. One finds several intriguing features, some of them discussed above and many of them unexplained. Two polynomials are grouped together if they define isomorphic graphs. 
The polynomials are arranged in lexicographical order, except for the fields of order $5^3$ and $5^4$ due to page layout reasons.
\begin{figure}[h!]
	\begin{minipage}[b]{0.5\textwidth} 
		\begin{center}
			\begin{tabular}{|c|c|c|} 
				\hline
				\multicolumn{2}{ |c| }{\scriptsize Irreducible monic polynomials} &  \\
				\multicolumn{2}{ |c| }{\scriptsize  with isomorphic graphs} &  \multirow{-2}{*}{\scriptsize Order of $\automorphism(X_f)$} \\
			 	\hline
					\multicolumn{3}{ |c| }{$2^2$}\\ 		
				\hline 
			 	\multicolumn{2}{ |r| }{$x^2 + x + 1$} & \multirow{1}{*}{2}\\
				\hline  
				 			\multicolumn{3}{ |c| }{$2^3$}\\
				\hline 
			 	\multicolumn{2}{ |r| }{$x^3 + x + 1$} & \multirow{1}{*}{144}\\
			 	\hline 
			 	\multicolumn{2}{ |r| }{$x^3 + x^2 + 1$} & \multirow{1}{*}{6}\\
				\hline 
							\multicolumn{3}{ |c| }{$2^4$}\\ 
				\hline 
			 	\multicolumn{2}{ |r| }{$x^4 + x + 1$} & \multirow{1}{*}{8}\\
			 	\hline 
			 	\multicolumn{2}{ |r| }{$x^4 + x^3 + 1$} & \multirow{1}{*}{4}\\
			 	\hline 
			 	\multicolumn{2}{ |r| }{$x^4 + x^3 + x^2 + x + 1$} & \multirow{1}{*}{4}\\
				\hline 
							\multicolumn{3}{ |c| }{$2^5$}\\ 
				\hline  
			 	\multicolumn{2}{ |r| }{$x^5 + x^2 + 1$} & \multirow{1}{*}{5}\\
			 	\hline 
			 	\multicolumn{2}{ |r| }{$x^5 + x^3 + 1$} & \multirow{1}{*}{5}\\
			 	\hline 
			 	\multicolumn{2}{ |r| }{$x^5 + x^3 + x^2 + x + 1$} & \multirow{1}{*}{5}\\
			 	\hline 
			 	\multicolumn{2}{ |r| }{$x^5 + x^4 + x^2 + x + 1$} & \multirow{1}{*}{5}\\
			 	\hline 
			 	\multicolumn{2}{ |r| }{$x^5 + x^4 + x^3 + x + 1$} & \multirow{1}{*}{5}\\
			 	\hline 
			 	\multicolumn{2}{ |r| }{$x^5 + x^4 + x^3 + x^2 + 1$} & \multirow{1}{*}{5}\\
			 	\hline
						    \multicolumn{3}{ |c| }{$3^2$}\\ 
				\hline 
			 	\multicolumn{2}{ |r| }{$x^2 + 1$} & \multirow{1}{*}{8}\\
			 	\hline 
			 	\multicolumn{2}{ |r| }{$x^2 + x + 2$} & \multirow{2}{*}{8}\\
				\multicolumn{2}{ |r| }{$x^2 + 2x + 2$}& \\
				\hline  
							\multicolumn{3}{ |c| }{$3^3$}\\ 
				\hline  
			 	\multicolumn{2}{ |r| }{$x^3 + 2x + 1$} & \multirow{1}{*}{6}\\
			 	\hline 
			 	\multicolumn{2}{ |r| }{$x^3 + 2x + 2$} & \multirow{1}{*}{6}\\
			 	\hline 
			 	\multicolumn{2}{ |r| }{$x^3 + x^2 + 2$} & \multirow{1}{*}{6}\\
			 	\hline 
			 	\multicolumn{2}{ |r| }{$x^3 + x^2 + x + 2$} & \multirow{1}{*}{6}\\
			 	\hline 
			 	\multicolumn{2}{ |r| }{$x^3 + x^2 + 2x + 1$} & \multirow{1}{*}{6}\\
			 	\hline 
			 	\multicolumn{2}{ |r| }{$x^3 + 2x^2 + 1$} & \multirow{1}{*}{6}\\
			 	\hline 
			 	\multicolumn{2}{ |r| }{$x^3 + 2x^2 + x + 1$} & \multirow{1}{*}{6}\\
			 	\hline 
			 	\multicolumn{2}{ |r| }{$x^3 + 2x^2 + 2x + 2$} & \multirow{1}{*}{6}\\
				\hline 
							\multicolumn{3}{ |c| }{$3^4$}\\ 
				\hline 
			 	\multicolumn{2}{ |r| }{$x^4 + x + 2$} & \multirow{1}{*}{8}\\
			 	\hline 
			 \end{tabular}
		 \end{center}
	\end{minipage}
	\begin{minipage}[b]{0.5\textwidth}
		\begin{center}
			 \begin{tabular}{|c|c|c|}
				\hline
				\multicolumn{2}{ |c| }{\scriptsize Irreducible monic polynomials} &  \\
				\multicolumn{2}{ |c| }{\scriptsize  with isomorphic graphs} &  \multirow{-2}{*}{\scriptsize Order of $\automorphism(X_f)$} \\
			 	\hline
			 	\multicolumn{2}{ |r| }{$x^4 + 2x + 2$} & \multirow{1}{*}{8}\\
			 	\hline 
			 	\multicolumn{2}{ |r| }{$x^4 + x^2 + 2$} & \multirow{2}{*}{512}\\
				\multicolumn{2}{ |r| }{$x^4 + 2x^2 + 2$} & \\
			 	\hline 
			 	\multicolumn{2}{ |r| }{$x^4 + x^2 + x + 1$} & \multirow{2}{*}{8}\\
				\multicolumn{2}{ |r| }{$x^4 + x^3 + x^2 + 1$} & \\
			 	\hline 
			 	\multicolumn{2}{ |r| }{$x^4 + x^2 + 2x + 1$} & \multirow{2}{*}{8}\\
				\multicolumn{2}{ |r| }{$x^4 + 2x^3 + x^2 + 1$} & \\
			 	\hline 
			 	\multicolumn{2}{ |r| }{$x^4 + x^3 + 2$} & \multirow{1}{*}{8}\\
			 	\hline
		 		\multicolumn{2}{ |r| }{$x^4 + x^3 + 2x + 1$} & \multirow{2}{*}{8}\\
				\multicolumn{2}{ |r| }{$x^4 + 2x^3 + x + 1$} & \\
		 		\hline 
		 		\multicolumn{2}{ |r| }{$x^4 + x^3 + x^2 + x + 1$} & \multirow{1}{*}{8}\\
		 		\hline 
			 	\multicolumn{2}{ |r| }{$x^4 + x^3 + x^2 + 2x + 2$} & \multirow{1}{*}{8}\\
			 	\hline 
			 	\multicolumn{2}{ |r| }{$x^4 + x^3 + 2x^2 + 2x + 2$} & \multirow{1}{*}{8}\\
				\hline
			 	\multicolumn{2}{ |r| }{$x^4 + 2x^3 + 2$} & \multirow{1}{*}{8}\\
			 	\hline 
			 	\multicolumn{2}{ |r| }{$x^4 + 2x^3 + x^2 + x + 2$} & \multirow{1}{*}{8}\\
			 	\hline 
			 	\multicolumn{2}{ |r| }{$x^4 + 2x^3 + x^2 + 2x + 1$} & \multirow{1}{*}{8}\\
			 	\hline 
			 	\multicolumn{2}{ |r| }{$x^4 + 2x^3 + 2x^2 + x + 2$} & \multirow{1}{*}{8}\\
			 	\hline
				\multicolumn{3}{ |c| }{$5^2$}\\ 
				\hline 
			 	\multicolumn{2}{ |r| }{$x^2 + 2$} & \multirow{2}{*}{16}\\
				\multicolumn{2}{ |r| }{$x^2 + 3$} & \\
			 	\hline 
			 	\multicolumn{2}{ |r| }{$x^2 + x + 1$} & \multirow{1}{*}{4}\\
			 	\hline 
			 	\multicolumn{2}{ |r| }{$x^2 + x + 2$} & \multirow{2}{*}{4}\\
				\multicolumn{2}{ |r| }{$x^2 + 3x + 3$} & \\
			 	\hline 
			 	\multicolumn{2}{ |r| }{$x^2 + 2x + 3$} & \multirow{2}{*}{4}\\
				\multicolumn{2}{ |r| }{$x^2 + 4x + 2$} & \\
			 	\hline 
			 	\multicolumn{2}{ |r| }{$x^2 + 2x + 4$} & \multirow{2}{*}{4}\\
				\multicolumn{2}{ |r| }{$x^2 + 3x + 4$} & \\
			 	\hline 
			 	\multicolumn{2}{ |r| }{$x^2 + 4x + 1$} & \multirow{1}{*}{4}\\
			 	\hline
				\multicolumn{3}{ |c| }{$5^3$}\\ 
				\multicolumn{3}{ |c|}{\small 40 polynomials with non-isomorphic graphs}\\[2pt]
				\multicolumn{3}{ |c|}{\small and an automorphism group of order 6.} \\ [2pt]
				 	\hline
			 \end{tabular}
		\end{center}
	\end{minipage}
\end{figure}

\newpage

\begin{figure}[h!]
	\begin{minipage}[b]{0.5\textwidth} 
  		\begin{center}
			\begin{tabular}{|c|c|c|}
				\hline 
				\multicolumn{2}{ |c| }{\scriptsize Irreducible monic polynomials} &  \\
				\multicolumn{2}{ |c| }{\scriptsize  with isomorphic graphs} &  \multirow{-2}{*}{\scriptsize Order of $\automorphism(X_f)$} \\
				\hline
				\multicolumn{3}{ |c| }{$5^4$}\\ 
				\hline 
			 	\multicolumn{2}{ |r| }{$x^4 + 2$} & \multirow{2}{*}{$\sim 3 \cdot 10^{47}$}\\
				\multicolumn{2}{ |r| }{$x^4 + 3$} & \\
			 	\hline 
			 	\multicolumn{2}{ |r| }{$x^4 + x^2 + 2$} & \multicolumn{1}{ c|}{\multirow{2}{*}{32768}}\\
				\multicolumn{2}{ |r| }{$x^4 + 3x^2 + 3$} & \\
				\hline
			 	\multicolumn{2}{ |r| }{$x^4 + 2x^2 + 3$} & \multicolumn{1}{ c|}{\multirow{2}{*}{32768}}\\
				\multicolumn{2}{ |r| }{$x^4 + 4x^2 + 2$} & \\
			 	\hline
			 	\multicolumn{2}{ |r| }{$x^4 + 2x^2 + 2x + 3$} & \multicolumn{1}{ c|}{\multirow{2}{*}{8}}\\
				\multicolumn{2}{ |r| }{$x^4 + 4x^3 + 4x^2 + 2$} & \\
			 	\hline
			 	\multicolumn{2}{ |r| }{$x^4 + 2x^2 + 3x + 3$} & \multicolumn{1}{ c|}{\multirow{2}{*}{8}}\\
				\multicolumn{2}{ |r| }{$x^4 + x^3 + 4x^2 + 2$} & \\
			 	\hline 
			 	\multicolumn{2}{ |r| }{$x^4 + 3x^2 + x + 3$} & \multicolumn{1}{ c|}{\multirow{2}{*}{8}}\\
				\multicolumn{2}{ |r| }{$x^4 + 2x^3 + x^2 + 2$} & \\
				\hline
			 	\multicolumn{2}{ |r| }{$x^4 + 3x^2 + 4x + 3$} & \multicolumn{1}{ c|}{\multirow{2}{*}{8}}\\
				\multicolumn{2}{ |r| }{$x^4 + 3x^3 + x^2 + 2$} & \\
				\hline
			 	\multicolumn{2}{ |r| }{$x^4 + x^3 + 2x + 4$} & \multicolumn{1}{ c|}{\multirow{2}{*}{8}}\\
				\multicolumn{2}{ |r| }{$x^4 + 3x^3 + 4x + 4$} & \\
			 	\hline
				\multicolumn{2}{ |r| }{$x^4 + x^3 + 4x + 1$} & \multicolumn{1}{ c|}{\multirow{2}{*}{32}}\\
				\multicolumn{2}{ |r| }{$x^4 + 4x^3 + x + 1$} & \\
			 	\hline
			 	\multicolumn{2}{ |r| }{$x^4 + x^3 + x^2 + 2x + 4$} & \multicolumn{1}{ c|}{\multirow{2}{*}{8}}\\
				\multicolumn{2}{ |r| }{$x^4 + 3x^3 + 4x^2 + 4x + 4$} & \\
			 	\hline 
			 	\multicolumn{2}{ |r| }{$x^4 + x^3 + x^2 + 3x + 3$} & \multicolumn{1}{ c|}{\multirow{2}{*}{8}}\\
				\multicolumn{2}{ |r| }{$x^4 + x^3 + 2x^2 + 2x + 2$} & \\
			 	\hline
			 	\multicolumn{2}{ |r| }{$x^4 + x^3 + x^2 + 4x + 2$} & \multicolumn{1}{ c|}{\multirow{2}{*}{8}}\\
				\multicolumn{2}{ |r| }{$x^4 + 2x^3 + 3x^2 + 3x + 3$} & \\
			 	\hline 
			 	\multicolumn{2}{ |r| }{$x^4 + x^3 + 2x^2 + x + 3$} & \multicolumn{1}{ c|}{\multirow{2}{*}{8}}\\
				\multicolumn{2}{ |r| }{$x^4 + 2x^3 + 4x^2 + 2x + 2$} & \\
			 	\hline 
			 	\multicolumn{2}{ |r| }{$x^4 + x^3 + 2x^2 + 3x + 4$} & \multicolumn{1}{ c|}{\multirow{2}{*}{8}}\\
				\multicolumn{2}{ |r| }{$x^4 + 2x^3 + 3x^2 + 4x + 4$} & \\
			 	\hline
			 	\multicolumn{2}{ |r| }{$x^4 + x^3 + 4x^2 + 4x + 1$} & \multicolumn{1}{ c|}{\multirow{2}{*}{8}}\\
				\multicolumn{2}{ |r| }{$x^4 + 4x^3 + 4x^2 + x + 1$} & \\
			 	\hline
			 	\multicolumn{2}{ |r| }{$x^4 + 2x^3 + x + 4$} & \multicolumn{1}{ c|}{\multirow{2}{*}{8}}\\
				\multicolumn{2}{ |r| }{$x^4 + 4x^3 + 3x + 4$} & \\
			 	\hline
			  	\multicolumn{2}{ |r| }{$x^4 + 2x^3 + x^2 + 3x + 1$} & \multicolumn{1}{ c|}{\multirow{2}{*}{8}}\\
				\multicolumn{2}{ |r| }{$x^4 + 3x^3 + x^2 + 2x + 1$} & \\
				\hline
			 	\multicolumn{2}{ |r| }{$x^4 + 2x^3 + 3x^2 + x + 2$} & \multicolumn{1}{ c|}{\multirow{2}{*}{8}}\\
				\multicolumn{2}{ |r| }{$x^4 + 3x^3 + 4x^2 + x + 3$} & \\
			 	\hline 
			 	\multicolumn{2}{ |r| }{$x^4 + 2x^3 + 4x^2 + x + 4$} & \multicolumn{1}{ c|}{\multirow{2}{*}{8}}\\
				\multicolumn{2}{ |r| }{$x^4 + 4x^3 + x^2 + 3x + 4$} & \\
			 	\hline 
			 	\multicolumn{2}{ |r| }{$x^4 + 2x^3 + 4x^2 + 4x + 3$} & \multicolumn{1}{ c|}{\multirow{2}{*}{8}}\\
				\multicolumn{2}{ |r| }{$x^4 + 3x^3 + 3x^2 + 4x + 2$} & \\
			 	\hline
			 	\multicolumn{2}{ |r| }{$x^4 + 3x^3 + 3x^2 + x + 4$} & \multicolumn{1}{ c|}{\multirow{2}{*}{8}}\\
				\multicolumn{2}{ |r| }{$x^4 + 4x^3 + 2x^2 + 2x + 4$} & \\
			 	\hline 
			 	\multicolumn{2}{ |r| }{$x^4 + 3x^3 + 3x^2 + 2x + 3$} & \multicolumn{1}{ c|}{\multirow{2}{*}{8}}\\ [1pt]
				\multicolumn{2}{ |r| }{$x^4 + 4x^3 + x^2 + x + 2$} & \\ [1pt]
				\hline
			 \end{tabular}
		\end{center}
	\end{minipage}
	\begin{minipage}[b]{0.5\textwidth}
		\begin{center}
 \begin{tabular}{|c|c|c|}
	\hline 
	\multicolumn{2}{ |c| }{\scriptsize Irreducible monic polynomials} &  \\
	\multicolumn{2}{ |c| }{\scriptsize  with isomorphic graphs} &  \multirow{-2}{*}{\scriptsize Order of $\automorphism(X_f)$} \\
 	\hline
 	\multicolumn{2}{ |r| }{$x^4 + 3x^3 + 4x^2 + 3x + 2$} & \multicolumn{1}{ c|}{\multirow{2}{*}{8}}\\
	\multicolumn{2}{ |r| }{$x^4 + 4x^3 + 2x^2 + 4x + 3$} & \\
 	\hline 
 	\multicolumn{2}{ |r| }{$x^4 + 4x^3 + x^2 + 2x + 3$} & \multicolumn{1}{ c|}{\multirow{2}{*}{8}}\\
	\multicolumn{2}{ |r| }{$x^4 + 4x^3 + 2x^2 + 3x + 2$} & \\
 	\hline 
 	\multicolumn{2}{ |r| }{$x^4 + x + 4$} & \multirow{1}{*}{8}\\
 	\hline 
 	\multicolumn{2}{ |r| }{$x^4 + 2x + 4$} & \multirow{1}{*}{8}\\
 	\hline 
 	\multicolumn{2}{ |r| }{$x^4 + 3x + 4$} & \multirow{1}{*}{8}\\
 	\hline
 	\multicolumn{2}{ |r| }{$x^4 + 4x + 4$} & \multicolumn{1}{ c|}{\multirow{1}{*}{8}} \\
  	\hline
 	\multicolumn{2}{ |r| }{$x^4 + x^2 + x + 1$} & \multicolumn{1}{ c|}{\multirow{1}{*}{8}}\\
 	\hline 
 	\multicolumn{2}{ |r| }{$x^4 + x^2 + 2x + 2$} & \multicolumn{1}{ c|}{\multirow{1}{*}{8}}\\
 	\hline 
 	\multicolumn{2}{ |r| }{$x^4 + x^2 + 2x + 3$} & \multicolumn{1}{ c|}{\multirow{1}{*}{8}}\\
 	\hline 
 	\multicolumn{2}{ |r| }{$x^4 + x^2 + 3x + 2$} & \multicolumn{1}{ c|}{\multirow{1}{*}{8}}\\
 	\hline 
 	\multicolumn{2}{ |r| }{$x^4 + x^2 + 3x + 3$} & \multicolumn{1}{ c|}{\multirow{1}{*}{8}}\\
 	\hline 
 	\multicolumn{2}{ |r| }{$x^4 + x^2 + 4x + 1$} & \multicolumn{1}{ c|}{\multirow{1}{*}{8}}\\
 	\hline
 	\multicolumn{2}{ |r| }{$x^4 + 2x^2 + 2x + 1$} & \multicolumn{1}{ c|}{\multirow{1}{*}{8}}\\
 	\hline 
 	\multicolumn{2}{ |r| }{$x^4 + 2x^2 + 3x + 1$} & \multicolumn{1}{ c|}{\multirow{1}{*}{8}}\\
 	\hline  
 	\multicolumn{2}{ |r| }{$x^4 + 3x^2 + x + 1$} & \multicolumn{1}{ c|}{\multirow{1}{*}{8}}\\
	\hline
 	\multicolumn{2}{ |r| }{$x^4 + 3x^2 + 4x + 1$} & \multicolumn{1}{ c|}{\multirow{1}{*}{8}}\\
 	\hline
 	\multicolumn{2}{ |r| }{$x^4 + 4x^2 + x + 2$} & \multicolumn{1}{ c|}{\multirow{1}{*}{8}}\\
 	\hline 
 	\multicolumn{2}{ |r| }{$x^4 + 4x^2 + x + 3$} & \multicolumn{1}{ c|}{\multirow{1}{*}{8}}\\
 	\hline 
	\multicolumn{3}{ |c|}{\small The remaining polynomials have}\\
	\multicolumn{3}{ |c|}{\small non-isomorphic graphs and} \\
	\multicolumn{3}{ |c|}{\small an automorphism group of order 8.} \\
	\hline 
	\multicolumn{3}{ |c| }{$7^2$}\\ 
 	\hline 
 	\multicolumn{2}{ |r| }{$x^2 + 1$} & \multirow{1}{*}{32}\\
 	\hline 
 	\multicolumn{2}{ |r| }{$x^2 + 2$} & \multirow{2}{*}{32}\\
	\multicolumn{2}{ |r| }{$x^2 + 4$} & \\
 	\hline 
 	\multicolumn{2}{ |r| }{$x^2 + x + 3$} & \multirow{2}{*}{4}\\
	\multicolumn{2}{ |r| }{$x^2 + 5x + 5$} & \\
 	\hline 
 	\multicolumn{2}{ |r| }{$x^2 + x + 4$} & \multirow{2}{*}{8}\\
	\multicolumn{2}{ |r| }{$x^2 + 2x + 2$} & \\
 	\hline 
 	\multicolumn{2}{ |r| }{$x^2 + x + 6$} & \multirow{2}{*}{4}\\
	\multicolumn{2}{ |r| }{$x^2 + 6x + 6$} & \\
 	\hline 
 	\multicolumn{2}{ |r| }{$x^2 + 2x + 3$} & \multirow{2}{*}{4}\\
	\multicolumn{2}{ |r| }{$x^2 + 3x + 5$} & \\
 	\hline 
 	\multicolumn{2}{ |r| }{$x^2 + 2x + 5$} & \multirow{2}{*}{4}\\
	\multicolumn{2}{ |r| }{$x^2 + 6x + 3$} & \\
 	\hline 
 	\multicolumn{2}{ |r| }{$x^2 + 3x + 1$} & \multirow{1}{*}{8}\\
 	\hline 
 	\multicolumn{2}{ |r| }{$x^2 + 3x + 6$} & \multirow{2}{*}{4}\\
	\multicolumn{2}{ |r| }{$x^2 + 4x + 6$} & \\
 	\hline 
 	\multicolumn{2}{ |r| }{$x^2 + 4x + 1$} & \multirow{1}{*}{8}\\
 	\hline 
 	\multicolumn{2}{ |r| }{$x^2 + 4x + 5$} & \multirow{2}{*}{4}\\
	\multicolumn{2}{ |r| }{$x^2 + 5x + 3$} & \\
 	\hline 
 	\multicolumn{2}{ |r| }{$x^2 + 5x + 2$} & \multirow{2}{*}{8}\\
	\multicolumn{2}{ |r| }{$x^2 + 6x + 4$} & \\
	\hline 
\end{tabular}
\end{center}
\end{minipage}
\end{figure}

Section de mathématiques, Université de Genève, 2-4 Rue du Lièvre,
Case Postale 64, 1211 Genève 4, Suisse 

e-mails: anders.karlsson@unige.ch, kuhn.gaetan@protonmail.ch

and

Matematiska institutionen, Uppsala universitet, Box 256, 751 05 Uppsala,
Sweden 

e-mail: anders.karlsson@math.uu.se
\end{document}